\documentclass[12pt]{article}
\usepackage{fancyhdr}

\usepackage{amsmath}
\usepackage{mathtools}
\usepackage{amsfonts}
\usepackage{amsthm}
\usepackage{amssymb}
\usepackage{color}
\usepackage{color}

\title{Existence to nonlinear parabolic problems with unbounded weights}

\usepackage{authblk}

\author[1]{Iwona Skrzypczak\thanks{email address: iskrzypczak@mimuw.edu.pl}}
\author[2]{Anna Zatorska--Goldstein\thanks{email address: azator@mimuw.edu.pl. The research of AZG has been supported
by~the~NCN
grant
no.~2012/05/E/ST1/03232
(years
2013-2017) and by the Foundation for Polish Science grant no. POMOST BIS/2012-6/3.}}
\affil[1]{\small
Institute of  Mathematics, Polish Academy of Sciences, \newline
ul. \'{S}niadeckich 8, 00-656 Warsaw, Poland
}
\affil[2]{\small
Institute of Applied Mathematics and Mechanics,
University of Warsaw, \newline
ul. Banacha 2, 02-097 Warsaw, Poland
}

\date{}
\begin{document}
\maketitle \sloppy

\thispagestyle{empty}

\renewcommand{\it}{\sl}
\renewcommand{\em}{\sl}

\newcommand{\barint}{
         \rule[.036in]{.12in}{.009in}\kern-.16in
          \displaystyle\int  }
\def\R{{\mathbb{R}}}
\def\r{{\mathbb{R}}}
\def\n{{\mathbb{N}}}
\def\rn{{\mathbb{R}^{N}}}
\def\rN{{\mathbb{R}^{N}}}
\def\rom{{\mathbb{R}^{m}}}
\def\w{\widetilde}
\def\M{{\cal M}}
\def\zi{[0,\infty )}
\def\sf{{\rm supp\,\phi}  }
\def\sx{{\rm supp\,\xi}  }
\def\rp{{\mathbb{R}_{+}}}
\def\bA{{\bar{A}}}
\def\O{{\Omega}}
\def\dv{{\rm div}}
\def\o{{\omega}}
\def\oo{{\overline{\omega}}}
\def\Oc{{\Omega\cap}}
\def\a{{\alpha}}
\def\b{{\beta}}
\def\e{{\epsilon}}
\def\d{{\delta}}
\def\D{{\Delta}}
\def\dL{{\Delta_p^{\omega_2}}}
\def\t{{\tau}}
\def\s{{\sigma}}
\def\g{{\gamma}}
\def\k{{\kappa}}
\def\tu{{\widetilde{u}}}
\def\tW{{\widetilde{W}}}
\def\na{{\nabla}}
\def\rr{{\rangle }}
\def\ll{{\langle }}
\def\irn{{\int_\rn}}

\newtheorem{theo}{\bf Theorem} 
\newtheorem{coro}{\bf Corollary}[section]
\newtheorem{lem}{\bf Lemma}[section]
\newtheorem{rem}{\bf Remark}[section]
\newtheorem{defi}{\bf Definition}[section]
\newtheorem{ex}{\bf Example}[section]
\newtheorem{fact}{\bf Fact}[section]
\newtheorem{prop}{\bf Proposition}[section]

\newcommand{\ol}{\overline}
\newcommand{\wt}{\widetilde}
\newcommand{\ck}{{\cal K}}
\newcommand{\ve}{\varepsilon}
\newcommand{\vp}{\varphi}
\newcommand{\pa}{\partial}
\newcommand{\psa}{\Phi_{s,A}}

%\dsp

\parindent 1em

\begin{abstract}

We consider the   weighted parabolic problem  of the type
\begin{equation*}
\begin{split}  \left\{\begin{array}{ll}
 u_t-\dv(\omega_2(x)|\nabla u|^{p-2} \nabla u )=  \lambda  \omega_1(x) |u|^{p-2}u,&  x\in\Omega,\\
 u(x,0)=f(x),& x\in\Omega,\\
 u(x,t)=0,& x\in\partial\Omega,\ t>0,\\
\end{array}\right.
\end{split}
\end{equation*}
for quite a general class of possibly unbounded weights $ \omega_1,\omega_2$ satisfying the Hardy-type inequality. We prove existence of a global weak solution in the weighted Sobolev spaces provided that $\lambda>0$ is smaller than the optimal constant in the inequality.

The domain is assumed to be bounded or quasibounded. The obtained solution is proven to belong to \[L^p(\R_+; W_{(\omega_1,\omega_2),0}^{1,p}(\Omega))\cap  L^\infty(\R_+; L^2(\Omega )).\] 
\end{abstract}

\smallskip



\section{Introduction}

It is already classical to involve the $p$--Laplace operator
\[
 \Delta_p  u =   {\dv}( |\nabla u|^{p-2}\nabla u)\] in modelling of various processes of diffusion-type (also interpreted in life or social sciences). We consider   the weighted operator called here $\omega$--$p$--Laplacian, which is defined as \begin{eqnarray} 
 \Delta_p^{\omega } u&=&  {\dv}(\omega(x)|\nabla u|^{p-2}\nabla u) \label{Lpom2}
\end{eqnarray}
 with a certain weight function $\omega:\Omega\to\R$. The meaning of replacing $ \Delta_p$ by  $\Delta^\omega_p$ would describe space non-homogeneity of the process.
 
 \medskip

Our aim is to provide clear, self-contained theory of existence  for    nonlinear parabolic equations of the type
\begin{equation}\label{paraprob0}
u_t-\dL u=\lambda \omega_1(x) |u|^{p-2}u\quad \text{in}\quad\Omega_T
\end{equation}
where $p>2$, weight functions $\omega_1,\omega_2\ge 0$ are possibly unbounded, $\Omega\subseteq\rn$ is~a~bounded open set, $\Omega_T=[0,T)\times\Omega$, $T>0$. We develop the previous results~\cite{isazg1} by~allowing $\omega_1$ to be unbounded, which entails  challenges in functional analysis of~the two-weighted Sobolev spaces $W_{(\omega_1,\omega_2)}^{1,p}(\Omega)$. See Section~\ref{ssec:weighted-setting} for the definition and comments on the functional setting. 

\medskip

We impose the restrictions on~the~weights in order to control the structure of~the~two--weighted Sobolev spaces, as well as to~ensure monotonicity of~the leading part of~the operator. Namely, we assume
\begin{itemize}
\item[(W1)] $\omega_1,\omega_2:\overline{\Omega}\to\R_+\cup\{0\}$ and $\omega_1,\omega_2\in L^{1}_{loc}(\Omega)$;

\item[(W2)] $\omega_1^{-\frac{2}{p-2}}\in L^1(\Omega)$;

\item[(W3)] for any  $U\subset\subset \Omega$ there exists a constant $\omega_2(x)\geq c_U>0$ in $U$;

\item[(W4)]  $(\omega_1,\omega_2)$ is a pair of weights in Hardy inequality 
\begin{equation}
\label{eq:hardyintro}
K \int_\Omega \ |\xi|^p \omega_{1}(x)\, dx \leqslant \int_\Omega |\nabla \xi|^p \omega_{2}(x)\,dx;
\end{equation} 

\end{itemize}

Furthermore, assume that there exists $s>p$ such that 
\begin{itemize}
\item[(W5)] for any  $U\subset\subset \Omega$ we have a compact embedding
\begin{equation*}%\label{intro:compemb}
W^{1,p}_{(\omega_1,\omega_2),0}(U)\subset\subset L^s_{\omega_1}(U);
\end{equation*}

\item[(W6)] there exists $q\in\left(p,s\right)$, such that \[\omega_1^{-\frac{q}{s-q}}\in L^1_{loc}(\Omega)\quad\text{ and} \quad \omega_2^{\frac{q}{q-p}}\in L^1_{loc}(\Omega).\] 
\end{itemize}

\bigskip

 Our main result is the following theorem.
\begin{theo}\label{theo:main} Let $p> 2$, $\Omega\subseteq\rn$ be an open subset, not necessarily bounded, and $f\in L^2(\Omega)$.  Assume   that $\omega_1,\omega_2$ satisfy conditions~(W1)--(W6). 

 Then there exists  $\lambda_0=\lambda_0(p,N,\omega_1,\omega_2)>0$
such that for all $\lambda\in(0,\lambda_0)$ and for arbitrary $T>0$, the parabolic problem
\begin{equation}\label{eq:main}\left\{\begin{array}{ll}
 u_t-\dL u=  \lambda \omega_1(x)|u|^{p-2}u & x\in\Omega_T,\\
 u(x,0)=f(x)& x\in\Omega,\\
 u(x,t)=0& x\in\partial\Omega,\ t>0,\\
\end{array}\right.
\end{equation}
has a weak solution \[u \in L^p(0,T; W_{(\omega_1,\omega_2),0}^{1,p}(\Omega)),\text{  such that }\ 
u_t  \in L^{p'}(0,T; {W^{-1,p'}_{(\omega_1',\omega_2')}(\Omega)}),\] i.e.
\begin{equation*}
%\label{eq:mainweak}
\int_{\Omega_T}\left(  u_t\xi+\omega_2|\nabla u|^{p-2} \nabla u \nabla \xi -\lambda \omega_1 |u|^{p-2}u \xi\right)dx\,dt=0,
\end{equation*}
holds for each $\xi\in L^p(0,T; W_{(\omega_1,\omega_2),0}^{1,p}(\Omega))$. 

Moreover, $u\in L^p(\R_+; W_{(\omega_1,\omega_2),0}^{1,p}(\Omega))\cap  L^\infty(\R_+; L^2(\Omega ))$. 
\end{theo}

 \begin{rem}In fact, the proof of the above theorem implies the existence to  \begin{equation*} \left\{\begin{array}{ll}
 u_t-\dL u=  \lambda W(x)|u|^{p-2}u & x\in\Omega_T,\\
 u(x,0)=f(x)& x\in\Omega,\\
 u(x,t)=0& x\in\partial\Omega,\ t>0,\\
\end{array}\right.
\end{equation*}
with any $W(x)\leq\omega_1(x)$ without the  assumption $||W||_{L^\infty(\Omega_T)}<\infty,$ which extends  the approach of~\cite{isazg1}.
\end{rem}

We present examples of weights satisfying the assumptions (W1)-(W6).   

\begin{rem}[Examples of admissible weights]\label{rem:ex} Denote $d(x)=dist(x,\partial\Omega)$. We call a domain $\Omega$ quasibounded if \[ \lim_{|x|\to\infty,\ x\in\Omega} d(x)=0.\] 
Suppose $\Omega\subset\rn$, $1\leq p\leq q<\infty$, $N>1$, $\frac{N}{q}-\frac{N}{p}+1>0$.  The conditions (W1)-(W6) are satisfied by the following types of weights:
\begin{itemize}
\item[i)] $\omega_1(x)=|x|^{-p}$, $\omega_2(x)\equiv 1$ on $\Omega$ being a bounded Lipschitz domain.  This example relates to the result of~\cite{gaap};
\item[ii)] $\omega_1(x)=d^{\beta-p}(x)$, $\omega_2(x)=d^{\beta}(x)$, with $\beta<p-N$, on $\Omega$ being bounded or quasibounded;  
\item[iii)] $\omega_1(x)=d^{\beta-p}(x)\,|\log d(x)|^\delta$, $\omega_2(x)=\d^{\beta }(x)\,|\log d(x)|^\delta$, close to the boundary  (when $d(x)\leq \frac{1}{2}$) and $\omega_1=\omega_2=const$ when $d(x)> \frac{1}{2}$, with $\beta<p-N<0,\,\delta>0$ on $\Omega$ being bounded or quasibounded.
\end{itemize}
We stress that in the cases {ii)} and {iii)} the domain can be unbounded as well.  
\end{rem}

\bigskip

\textbf{State of art.} The existence of solutions to problems
\[u_t-\dv A(x,t,u,\nabla u)= f,\]
 where the involved operator is monotone and has $p$--growth, is very well understood, e.g.~\cite{boc-ors1,boc-ors2,boccardo}. Nonetheless, this research concerns the case, when the right--hand side does not depend on~the~solution itself.   Various physical models (\textit{combustion models}) involve semilinear parabolic problems of the form
$$
u_t-\Delta u= f(u).
$$
Fujita's Theory, developed since 1960s, analyses the possible singularities of~solutions. There are known examples of problems, where solutions explode (blow-up) to infinity in finite time. More recent research in that directions was carried out by Giga and Kohn.

In~\cite{vz} Vazquez and Zuazua, generalizing the seminal paper by Baras and Goldstein~\cite{bargold}, describe the asymptotic behaviour of the heat equation that reads \[u_t=\D u +V(x)u \quad\mathrm{and}\quad \D u +V(x)u+\mu u=0,\]
where $V(x)$ is an inverse--square potential. 

% This phenomenon is observed in wide range of parabolic problems, including semilinear equations, see e.g.~\cite{AdChaRa,anh,bargold,n,BV,bh,32,gaap,xiang,vz}.  In several papers, e.g.~\cite{bbdg,blanchet_09,sharp}, dealing with the rate of convergence of~solutions to fast diffusion equations $u_t=\Delta u^m,$ the authors study the estimates for the constants in Hardy-Poincar\'{e}-type inequalities and their application.   
 
%In general, application of the general Hardy inequalities is expected to infer certain properties of solution to wide class of~parabolic problems.

 The inspiration of our research was the paper of  Garc\'{i}a Azorero and  Peral Alonso~\cite{gaap}, who %apply the Hardy inequality  {\cite[Lemma 2.1]{gaap}} of the form
%\footnote{Compare with Corollary~\ref{coroharrn}: $\int_{\rn\setminus\{0\}} \ |\xi|^p |x|^{\g-p}  \ dx \le \widetilde{C}_{min}\int_{\rn\setminus\{0\}} |\nabla \xi|^p{|x|^{\g}}\, dx.$}
%\[\lambda_{N,p}\int_{\rN} \ |\xi|^p |x|^{-p}  \ dx \le  \int_{\rN} |\nabla \xi|^p\, dx ,\]
%where $\lambda_{N,p}$ is optimal, to 
obtain the existence of weak solutions to the corresponding parabolic problem \begin{equation*}   u_t -\Delta_p u=\frac{\lambda}{|x|^p}|u|^{p-2}u%,%\qquad 1<p <N
\end{equation*}
on a bounded domain~$\Omega$. The topic was developed  by
 Dall'Aglio,  Giachetti, and Peral~\cite{n} concern analysis of solutions to the problem of the form
\[u_t -\dv\left( \frac{|\nabla u|^{p-2}\nabla u}{|x|^{p\gamma}}\right)=\frac{\lambda |u|^{p-2}u}{|x|^{(\gamma+1)p}}\]
considered on a bounded domain~$\Omega$. For further closely related results we refer to~\cite{AbPeWa,AtMePe,GoHaRh,MeMoPeSc}, as for the weighted fast diffusion equation to~\cite{weightmatteo1,weightmatteo2}.

Another approach that is influential for us comes from~\cite{anh} by Anh and~Ke. The  initial boundary value problem for a class of~quasilinear parabolic equations  considered therein involves, as in our case~\eqref{paraprob0}, the weighted $p$-Laplace operator and reads
\[u_t-\dv(\sigma(x)|\nabla u|^{p-2} \nabla u )=  \lambda  |u|^{p-2}u-f(x,u),%,\qquad 2\leq p <N
\] where $f$ is a function of power-type growth with respect to the second variable perturbed additively by integrable dependence on the first variable.

Although there exist vast literature on existence to parabolic problems, it seems particularly hard to find a proper comprehensive reference to studies on problems in the weighted setting except the case of power-type weights.

\bigskip

\textbf{Our approach.} We analyse nonlinear problems of the type~\eqref{paraprob0} in the two-weighted spaces, where the involved weights are general. In Preliminaries we explain in detail notation, properties of the weighted setting, as well as the role of each of assumptions (W1)-(W6). For this moment let us only mention few key objectives. Our major difficulties result from involving more advanced setting than the classical one investigated in~\cite{anh,n,gaap}. 
 Due to presence of~general class of weights both in the leading part of~the operator and on~the~right-hand side of~\eqref{paraprob0}, we employ the two-weighted Sobolev spaces $W^{1,p}_{(\omega_1,\omega_2)}(\Omega)$.  Since the weights are different, integration by parts can be non-admissible  and the structure of~the dual space complicates.  Indeed, in the case of the weighted Lebesgue spaces we observe that $(L^p_{\omega }(\Omega))^*\neq  L^{p'}_{\omega }(\Omega)$. We shall stress that unbounded domains are admissible, if only (W5) holds. Remark~\ref{rem:ex} provides examples of such domains.

Let us concisely summarize the main ideas of the proof. In order to construct a  weak solution to~\eqref{eq:main} we first consider a sequence of problems with truncated weights on the right-hand side, to which the solutions exist due to {\cite[Theorem~3.1]{isazg1}}. Then we pass to the limit with the level of truncation using the auxiliary compactness results inspired by the results introduced in the non-weighted $p$-Laplacian case by~\cite{boccardo,bocmurpuel}.  
 
 \bigskip
 
The paper is organised as follows. Section~\ref{prelim} provides disscusion on properties of the two-weighted Sobolev spaces and assumptions on the admissible weights.  Auxiliary compactness results are presented in Section~\ref{sec:aux}, whereas in Section~\ref{sec:main} the proof of Theorem~\ref{theo:main} is given. In the end we attach Appendix providing the required classical tools.

\section{Preliminaries}\label{prelim}

\subsection{Notation}

In the sequel we assume that  $p>  2$, $\frac{1}{p}+\frac{1}{p'}=1$, $\O\subset \rn$ is an open subset not necessarily bounded. For $T>0$ we denote $\Omega_T=\Omega\times(0,T)$. By $\langle f,g\rangle$ we denote the standard scalar product in $L^2(\Omega)$.

Let $B(r)\subset\rn$ denote the ball with the radius $r$, whose center shall be clear from the context. Then $|B(r)|$ is its Lebesgue's measure, while $\omega(B(r))$ its $\omega$-measure, i.e. $\omega(B(r))=\int_{B(r)}\omega(x)\,dx$.

We use truncations $T_k(f)(x)$ defined as follows \begin{equation}T_k(f)(x)=\left\{\begin{array}{ll}f & |f|\leq k;\\
k\frac{f}{|f|}& |f|\geq k.
\end{array}\right. .\label{Tk}
\end{equation}

\subsection{Weighted Lebesgue and Sobolev spaces}\label{ssec:weighted-setting}

Suppose $\omega$ is a positive, Borel  measurable, real  function  defined on an open set $\Omega\subset  \rn$. Let   \begin{equation}
 \label{om'}
 \omega'=\omega^{-1/(p-1)}.
 \end{equation} 
 
\begin{defi}[$B_p$--condition,~\cite{kuf-opic}] We say that $\omega$ satisfies the $B_p$--condition on $\Omega$ ($\omega\in B_p(\Omega)$), if
 \begin{equation}
 \label{Bp}
 \omega'\in L^1_{{  loc}}(\Omega).
 \end{equation}
\end{defi}

Note that any $\omega\in L^1_{loc}(\Omega)$, which is strictly positive inside $\Omega$ satisifes $B_p$ condition on $\Omega$.
 
\begin{rem} When $1<p<\infty$ and $\omega\in B_p$, we have $\displaystyle L^{p}_{\omega,loc}(\Omega)\subseteq  L^1_{{ loc}}(\Omega),$ see~\cite{kuf-opic}. Moreover, for any $\omega\in L^1_{loc}(\Omega)$ and $s>p$ we have
\begin{equation}
\label{emb:sp}
L^s_{\omega,loc}(\Omega)\subset L^p_{\omega,loc}(\Omega).
\end{equation}
\end{rem}
 
\bigskip

 If $\nabla$ denotes distributional gradient% and $\omega_1,\omega_2$ satisfy $B_p$--condition~\eqref{Bp}
 , we denote
\begin{equation}\label{polnorma}
W^{1,p}_{(\omega_1,\omega_2)}(\Omega):= \left\{ f\in L^{p}_{\omega_1}(\Omega) :  %\frac{\partial f}{\partial x_1} , \dots ,\frac{\partial f}{\partial x_n}
\nabla f\in (L_{\omega_2}^p (\Omega ))^N\right\}
\end{equation}
  with the norm
\begin{multline*}
\| f\|_{W^{1,p}_{(\omega_1,\omega_2)}(\O)}\ :\,=\ \| f\|_{L^{p}_{\omega_1}(\O)} + \| \nabla f\|_{(L_{\omega_2}^p (\Omega ))^N}\\
=\left(\int_{\Omega}|f|^{p}{\omega_1(x)}dx \right)^\frac{1}{p} + \left(\int_{\Omega}\sum_{i=1}^N\left|\frac{\partial f}{\partial x_i} \right|^{p}{\omega_2(x)}dx \right)^\frac{1}{p}.
\end{multline*}
 
 Under $B_p$-condition, the weighted Sobolev space has the basic properties.
 \begin{fact}[e.g.~\cite{kuf-opic}]\label{factemb}   If  $p>1$, $\Omega\subset \rn$ is an open set, $\omega_1,\omega_2$ satisfy $B_p$--condition~\eqref{Bp}, then \begin{itemize}
 \item[(i)] $W^{1,p}_{(\omega_1,\omega_2)}(\Omega)$ defined by \eqref{polnorma} equipped with the norm $\| \cdot\|_{W_{(\omega_1,\omega_2)}^{1,p}(\Omega)}$ is a~Banach space;
  \item[(ii)]  $\displaystyle \overline{Lip_0(\Omega)}  =\overline{C^\infty_0(\Omega)}  =W^{1,p}_{(\omega_1,\omega_2),0}(\O),$ where the closure is in the norm $\| \cdot\|_{W_{(\omega_1,\omega_2)}^{1,p}(\Omega)}$;
\item[(iii)] if $\omega_1,\omega_2$ are a pair in the Hardy-Poincar\'{e} inequality of the form~\eqref{eq:hardyintro}, we may consider the Sobolev space $W^{1,p}_{(\omega_1,\omega_2),0}(\O)$ equipped with the norm
\[\| f\|_{W^{1,p}_{(\omega_1,\omega_2), 0}(\O)} = \| \nabla f\|_{L^{p}_{\omega_2}(\O)}.\]
  \end{itemize} 
\end{fact}

\begin{fact}%\label{facthemi} 
Operator $\dL$, given by~\eqref{Lpom2}, is hemicontinuous, i.e. for all $u,v,w\in W^{1,p}_{(\omega_1,\omega_2),0}(\O)$ the mapping $\lambda\mapsto\ll \dL(u+\lambda v),w\rr$ is continuous from $\r$ to $\r$.
\end{fact}

\bigskip

We look for solutions in the space $L^p(0,T;W^{1,p}_{(\omega_1,\omega_2)}(\O))$, i.e.
\begin{equation*}
%\label{LpWp}
L^p(0,T;W^{1,p}_{(\omega_1,\omega_2)}(\O))=\left\{f\in L^p(0,T;L^{p}_{\omega_1}(\O)):\nabla f\in (L^p(0,T;L^{p}_{\omega_2}(\O)))^N\right\},
\end{equation*}
where $\nabla$ denotes distributional gradient with respect to the spacial variables, equipped with the norm
\begin{equation*}
%\label{normLpWp}
\| f\|_{L^p(0,T;W^{1,p}_{(\omega_1,\omega_2)}(\O))}:= \left(\int_{0}^T\| f\|^p_{L^{p}_{\omega_1}(\O)}dt\right)^\frac{1}{p} + \left(\int_{0}^{T}\| \nabla f\|^p_{(L_{\omega_2}^p (\Omega ))^N}dt\right)^\frac{1}{p}.
\end{equation*}

\subsubsection*{Dual spaces}

Let us stress that \[(L^p_{\omega }(\Omega))^*\neq  L^{p'}_{\omega }(\Omega) ,\qquad \mathrm{but}\qquad (L^p_{\omega }(\Omega))^*= L^{p'}_{\omega '}(\Omega) \]
with $\omega'$ given by~\eqref{om'}. 

By $W^{-1,p'}_{(\omega_1',\omega_2')}(\O)$ we denote the dual space to $W^{1,p}_{(\omega_1,\omega_2),0}(\O)$ and the duality pairing is given by the standard scalar product. We note that \[ L^{p'}(0,T;W^{-1,p'}_{(\omega_1',\omega_2')}(\O))\quad\text{is the dual space to}\quad L^p(0,T;W^{1,p}_{(\omega_1,\omega_2),0}(\O)).\]

 \subsection{Comments on admissible weights}\label{ssec:comwei}

Let us  present the reasons for which we assume the conditions (W1)-(W6). 

Condition (W1) is a general assumption fixing the setting of $L^{p}_{\omega_1}(\Omega)$ and $W^{1,p}_{(\omega_1,\omega_2)}(\Omega)$.  To ensure that the weighted  Sobolev space $W^{1,p}_{(\omega_1,\omega_2)}(\Omega)$ is a Banach space, we need to assume $\omega_1\in B_p(\Omega)$, cf.~\eqref{Bp}. However, it is necessary to assume a stronger condition  $\omega_1^{-\frac{2}{p-2}}\in L_{loc}^1(\Omega)$, to obtain the embbedding
\[L^p_{\omega_1,loc}(\Omega)\subset L^{p'}_{\omega_1',loc}(\Omega)\]and (W2), namely $\omega_1^{-\frac{2}{p-2}}\in L ^1(\Omega)$, to ensure \[ W^{1,p}_{(\omega_1,\omega_2),0}(\Omega)\subset L^2(\Omega).\]
 Condition (W3)  guarantees strict monotonicity of the operator.  Moreover, it implies that $\omega_2\in B_p(\Omega)$, cf.~\eqref{Bp}, which is necessary to ensure that  $W^{1,p}_{(\omega_1,\omega_2)}(\Omega)$ is a Banach space. 
 
We need $(\omega_1,\omega_2)$ to be a pair of weights in the Hardy inequality (W4), to get Fact~\ref{factemb} {\it (iii)}.  We shall stress that there are multiple methods of deriving weights admissible in the Hardy inequalities having the form~\eqref{eq:hardyintro}. In particular, the results of the first author~{\cite[Theorem 4.1]{plap}} show that the weights  may be generated by~nonnegative solutions to the elliptic problem and the regularity conditions imposed on the weights are in fact expected regularity properties of~the~solutions. 
 
Condition (W5), namely a compact embedding $W^{1,p}_{(\omega_1,\omega_2),0}(U)\subset\subset L^s_{\omega_1}(U)$ for any $U\subset\subset \Omega$, is necessary for the compactness method of Boccardo and Murat~\cite{boccardo}. To obtain (W5) the result by Franchi, Serapioni and Serra Cassano~\cite[Theorem~3.4]{fra-ser-ser} can be applied. If one is equipped with another continuous embedding of the weighted Sobolev space into the weighted Lebesgue space, they may apply the results by Opic and Kufner~\cite[Sections 17 and 18]{OpKuf} to obtain compact embedding on domains similar to John domains.  In particular, the authors provide the Muckenhoupt-type conditions for radial weights on an outer domain sufficient for compactness of the required embedding. In the approach of~\cite{fra-ser-ser} the condition is also of the Muckenhoupt-type and the crucial issue is geometry of the boundary. For other ideas on compact embeddings in weighted Sobolev spaces we refer to~\cite[Proposition~2.1]{anh} by Anh and Ke.

In the end we assume the technical integrability condition~(W6). Note that $\omega_1^{-q/({s-q})}\in L^1_{loc}(\Omega)$ may follow from (W2). It depends on the possible values of exponents~$s$ and $q$. Notice its consequence in~\eqref{embLq}. 

\bigskip

We have the following consequences of the assumptions  on embeddings.
 
\begin{rem}%[Continuous embeddings]%\label{rem:embobv}
If $\Omega$ is bounded, $2\leq p<\infty$ and $\omega_1,\omega_2$ satisfy~(W1)-(W4), then\begin{equation*}
%\label{emb:chain} 
W^{1,p}_{(\omega_1,\omega_2)}(\Omega)\subset 
  L^{p'}_{\omega_1'}(\Omega)=( L^p_{\omega_1}(\Omega))^*
 \subset (W^{1,p}_{(\omega_1,\omega_2),0}(\Omega))^*=W^{-1,p'}_{(\omega_1',\omega_2')}(\Omega).
\end{equation*} 
and
\begin{equation*}
 L^p(0,T;
W^{1,p}_{(\omega_1,\omega_2)}(\Omega)) \subset L^p(0,T;L^{p'}_{\omega_1'}(\Omega))\subset  L^{p'}(0,T;W^{-1,p'}_{(\omega_1',\omega_2')}(\Omega)).
\end{equation*}
\end{rem}

\begin{rem} If $\Omega$ is bounded, $2< p<s $ and $\omega_1,\omega_2$ satisfy~(W1)-(W5), then \begin{equation}
\label{emb:chain2} W^{1,p}_{(\omega_1,\omega_2)}(\Omega)\subset \subset
 L^s_{\omega_1}(\Omega)\subset L^{p'}_{\omega_1'}(\Omega) 
 \subset W^{-1,p'}_{(\omega_1',\omega_2')}(\Omega).
\end{equation}%\begin{equation}
%\label{emb:non} W^{1,p}_{(\omega_1,\omega_2)}(\Omega)\subset \subset L^s_{\omega_1}(\Omega)\subset L^{p'}_{\omega_1'}(\Omega)  \subset W^{-1,p'}_{(\omega_1',\omega_2')}(\Omega),
%\end{equation}
Furthermore, \begin{equation*}
%\label{eq:L2emb}
W^{1,p}_{(\omega_1,\omega_2),0}(\Omega)\subset\subset L^2(\Omega)
\end{equation*}
and \begin{equation}
\label{eq:L2L2emb}
L^p(0,T;W^{1,p}_{(\omega_1,\omega_2),0}(\O))\subset  L^2(0,T;L^2(\Omega))= L^2(\Omega_T).
\end{equation}
Moreover, if additionally we have (W6), then
\begin{equation}
\label{embLq}  
 L^s_{\omega_1,loc}(\Omega)\subset L^{q}_{loc} (\Omega)\quad \text{for $q\in(p,s)$}.
\end{equation}
\end{rem}

\section{Auxiliary results}\label{sec:aux}
 
 This section  concerns necessary compactness properties and recalls the result on existence to the problem with a bounded weight on the right-hand side.
 
 \medskip
  
We need the following version of~\cite[Lemma~4.2]{boccardo} adjusted to the weighted setting. \begin{prop}\label{prop:boc-mur-appl} Suppose $2<p<\infty$ and $\omega_1,\omega_2$ satisfy~(W1)-(W5). Assume further that 
\begin{equation}\label{eq:lem:boc-mur}
(u_m)_t=h_m\quad \text{in} \quad {\cal D}'(\Omega),
\end{equation}
where $h_m$ --- bounded in $L^{p'}(0,T; W^{-1,p'}_{(\omega_1,\omega_2)}(U))$ and $u_m\xrightharpoonup[m\to\infty]{} u$ in~$L^{p}(0,T; W^{1,p}_{(\omega_1,\omega_2),0}(\Omega)).$

Then  \begin{itemize}
\item[(a)] $u_m\xrightarrow[m\to\infty]{} u$ strongly in $L^{p}(0,T; L^{s}_{\omega_1}(U));$
\item[(b)]  $u_m\xrightarrow[m\to\infty]{} u$ a.e. in $\Omega_T$ (up to a subsequence).
\end{itemize}
\end{prop}
 
\begin{proof}%[Proof of Theorem~\ref{theo:boc-mur-appl}] 
 Let us consider a function $\phi(x,t)=\psi(x)\eta(t)$, where $\psi\in {\cal D}(\Omega)$ and $\eta\in  {\cal D}(0,T)$, and set $v_m=\phi u_m$. For any bounded open subset $U$, such that $\mathrm{supp}\phi\subset U\subset\Omega, $ we have \[(v_m)_t=(\phi u_m)_t=\phi(u_m)_t+\phi_tu_m=\phi h_m+\phi_tu_m.\]

Then $v_m$ is bounded in $L^{p}(0,T; W^{1,p}_{(\omega_1,\omega_2)}(U))$  and, due to~\eqref{eq:lem:boc-mur}, $(v_m)_t$ is bounded in $L^{p'}(0,T; W^{-1,p'}_{(\omega_1',\omega_2')}(\Omega)).$ We are going to apply the Aubin--Lions Lemma (Theorem~\ref{AubinLionsLemma}). Let us note that if $p>2$, then (W5) and~\eqref{emb:chain2} gives
\[W^{1,p}_{(\omega_1,\omega_2),0}(U)\subset\subset L^s_{\omega_1}(U) \subset W^{-1,p'}_{(\omega_1',\omega_2')}(\Omega).\]
Therefore $v_m$ is relatively compact in $L^p(0,T;L^s_{\omega_1}(U)).$ 
 
Moreover, since we know~\eqref{eq:L2L2emb}, strong convergence in Lebesgue's space implies convergence almost everywhere. 
\end{proof} 
 
For the convenience of the reader, we provide the  following extension of~\cite[Lemma~5]{bocmurpuel} with the proof.
 \begin{prop}\label{prop:bmpuel} Let $U$ be a bounded open subset in $\rn$, $U_T:=U\times(0,T)$, $2<p<\infty$ and $\omega_1,\omega_2$ satisfy~(W1)-(W4).  Assume that  $\nu_m\rightharpoonup \nu$ weakly in~$L^{p}(0,T; W^{1,p}_{(\omega_1,\omega_2),0}(U))$ and a.e. in $U_T$,  and \begin{equation}
 \int_{U_T} \omega_2 \left[  |\nabla \nu_m|^{p-2}\nabla \nu_m- |\nabla \nu |^{p-2}\nabla \nu \right]\nabla(\nu_m-\nu)\,dx\,dt\to 0.\label{eq:limzero}
 \end{equation}
 Then $\nabla\nu_m\to \nabla\nu$ strongly in $L^{p}(0,T; (L^{p}_{\omega_2}(U))^N)$, when $m\to\infty$.
 \end{prop}

 \begin{proof}%[Proof of~Theorem~\ref{thm:bmpuel}]
Let $D_m$ be defined by 
 \[D_m(x)= \left[  |\nabla \nu_m|^{p-2}\nabla \nu_m- |\nabla \nu |^{p-2}\nabla \nu \right]\nabla(\nu_m-\nu).\]
 By the monotonicity of $\Delta_p^{\omega_2}$ we note that $\omega_2(x)D_m\geq 0$. Since~\eqref{eq:limzero}, observe that $D_m\to 0$ in $L^1(0,T;L^1_{\omega_2}(U))$ strongly. Thus, up to a subsequence $D_m\to 0$ a.e. in~$U_T$. Recall $U_T$ is bounded. Suppose $X\subset U$ is a maximal set of full Lebesgue's measure (and therefore of full $\omega_2$-measure), where for each $x\in X$ we have
 \[
 |\nu(x)|<\infty,\quad |\nabla \nu(x)|<\infty,\quad \nu_m(x)\to \nu(x),\quad D_m(x)\to 0.\]
 Clearly $\omega_2|\nabla \nu_m|^p\geq 0$ and $0\leq D_m(x)$. Moreover,
 \begin{multline*} D_m(x)=|\nabla \nu_m|^p+|\nabla \nu|^p-|\nabla \nu_m|^{p-2}\nabla \nu_m \nabla \nu-|\nabla \nu |^p \nabla \nu  \nabla \nu_m\geq \\
 \geq |\nabla \nu_m|^p -c(x)\left(|\nabla \nu_m|^{p-1}+|\nabla \nu_m|\right),
 \end{multline*}
with $c(x)$ dependent on $X$, but not on $m$. As $D_m(x)\to 0$, we infer that $|\nabla \nu_m|$ is uniformly bounded on $X$.

Let us take arbitrary $x_0\in X$ and denote
\[\zeta_m=\nabla \nu_m(x_0),\qquad \zeta =\nabla \nu (x_0).\]
Observe that $\omega_2(x_0)>0$ and $(\zeta_m)$ is a bounded sequence. Set $\zeta_*$ as one of~its cluster points. Recall $D_m(x_0)\to 0$ and note that
\[ D_m(x_0)\to (|\zeta_*|^{p-2}\zeta_*-|\zeta |^{p-2}\zeta )(\zeta_*-\zeta).\]
Thus, $\zeta=\zeta_*$ is a unique cluster point of whole the sequence and $\nabla \nu_m(x_0)\to\nabla \nu (x_0)$ for arbitrary $x_0\in X$. Then
\[\omega_2|\nabla \nu_m|^p\to \omega_2|\nabla \nu |^p\qquad \mathrm{in}\quad X.\]
It implies uniform integrability of the sequence $  |\nabla u_m|^{p } $ in  $L^1_{\omega_2}(X)$, which implies uniform integrability in  $L^1_{\omega_2}(U)$.
 
 Therefore, Vitali's Convergence Theorem (Theorem~\ref{theo:VitConv}) yields that 
 \[ \int_U \omega_2  \left(  |\nabla \nu_m|^{p } - |\nabla \nu |^{p } \right)dx\to 0\quad \text{for}\quad m\to\infty \]and the claim follows.\end{proof}

We use also the following modification of~\cite[Theorem~4.1]{boccardo}. \begin{prop}\label{prop:boc-mur}Assume  $2<p<\infty$, $\omega_1,\omega_2$ satisfy~(W1)-(W6). Suppose
\begin{equation}\label{eq:prob:boc-mur}
(u_m)_t-\dL (u_m)=g_m\qquad in\ {\cal D}'(\Omega),
\end{equation}
and $g_m\to g$ in $L^{p'}(0,T; W^{-1,p'}_{(\omega_1',\omega_2')}(\Omega))$ and $u_m\xrightharpoonup {} u$ in~$L^{p}(0,T; W^{1,p}_{(\omega_1,\omega_2),0}(\Omega)),$ when $m\to\infty$.

 Then, for any fixed $k>0$, we have the strong convergence of  the gradients
 \begin{equation} \label{end} \nabla T_k(u_m)\xrightarrow[m\to\infty]{} \nabla T_k(u)\qquad in\quad L^{p}\left(0,T; (L^{p}_{\omega_2}(U))^N\right).\end{equation}
\end{prop} 

\begin{proof} To get the  strong convergence  of the gradients it suffices to prove that  \begin{equation}\begin{split}\label{Em}
E^m:=\int_{\Omega_T}  \phi_K \omega_2\left[|\nabla T_k(u_m)|^{p-2}\nabla T_k(u_m)-|\nabla T_k(u)|^{p-2}\nabla T_k(u)\right]\cdot\\
\cdot  [\nabla T_k(u_m)-\nabla T_k(u)]dxdt\xrightarrow[m\to\infty]{}0\end{split}
\end{equation}   Indeed, due to weak convergence $u_m\xrightharpoonup{} u$ in $L^{p}(0,T; W^{1,p}_{(\omega_1,\omega_2),0}(\Omega))$ we can apply Proposition~\ref{prop:boc-mur-appl} get $u_m\xrightharpoonup{} u$ a.e.  in $\Omega_T$. Then Proposition~\ref{prop:bmpuel} for $\nu=T_k(u)$ and $\nu_m=T_k(u_m)$ will give~\eqref{end}. 

\medskip

To get~\eqref{Em} we write
\begin{equation}\label{Emsplit}\begin{split}E^m=&-\int_{\Omega_T}   \phi_K\omega_2 |\nabla  u_m |^{p-2}\nabla  u_m   [\nabla T_k(u_m)-\nabla T_k(u)]\chi_{\{u_m>k\}}dxdt+\\
&-\int_{\Omega_T}  \phi_K \omega_2 |\nabla T_k(u)|^{p-2}\nabla T_k(u)  [\nabla T_k(u_m)-\nabla T_k(u)]dxdt=\\
&+\int_{K}   \phi_K\omega_2 |\nabla  u_m |^{p-2}\nabla  u_m   [\nabla T_k(u_m)-\nabla T_k(u)]dxdt=\\
=&E_1^m+E_2^m+E_3^m,\end{split}\end{equation}
where we will show that each of $E_1^m,E_2^m,E_3^m$ converges to zero when $m\to\infty$. Since $\nabla T_k(u_m) \chi_{\{u_m>k\}}=0$, we have 
 \[E_1^m=\int_{\Omega_T}   \phi_K\omega_2 |\nabla  u_m |^{p-2}\nabla  u_m    \nabla T_k(u) \chi_{\{u_m>k\}}dxdt,\]
 where $|\nabla  u_m |$ is bounded in $L^{p }(0,T;(L^{p }_{\omega_2}(\Omega))^N) $ and for $m\to\infty$ we have $\nabla T_k(u) \chi_{\{u_m>k\}}\to \nabla T_k(u)\chi_{\{u >k\}}$ strongly in $L^{p}(0,T;L^{p}_{\omega_2}(U))^N).$ Then the~Monotone Convergence Theorem and fact that $u_m$ is nondecreasing give the point. $E_2^m$ converges to zero, because \[T_k(u_m)-T_k(u)\xrightharpoonup[m\to\infty]{} 0\quad \mathrm{weakly\ in\ }L^{p}(0,T; W^{1,p}_{(\omega_1,\omega_2)}(\Omega)).\]
 
Therefore, it suffices to prove that $E_3^m\to 0$ when $m\to\infty$. Let us concentrate on~\eqref{eq:prob:boc-mur} tested by a proper choice of test function. We define $S_k(s)=\int_0^s T_k(r)\,dr$, where $T_k$ is given by~\eqref{Tk}. Then for any $\phi\in{\cal D}(\Omega_T)$ and any $\zeta\in L^{p}(0,T;W^{1,p}_{(\omega_1,\omega_2)}(\Omega))$ such that $\zeta_t\in L^{p'}(0,T;W^{-1,p'}_{(\omega_1,\omega_2)}(\Omega))$ we have \[\int_{\Omega_T} \zeta_t\phi T_k(\zeta)dxdt=-\int_{\Omega_T} \phi_t\,S_k(\zeta)dxdt.\]
We fix arbitrary compact sets $K\subset \Omega_T$ and $U\subset \Omega$, such that $K\subset(0,T)\times U \subset \Omega_T$ and take an arbitrary function $\phi_K\in{\cal D}(\Omega_T)$ with $\mathrm{supp}\,\phi_K\subset K\subset\subset \Omega_T,$ such that $0\leq \phi_K\leq 1$   with $\phi_K=1$ on $K$. Then we test~\eqref{eq:prob:boc-mur} by
\[w_m=(T_k(u_m)-T_k(u))\phi_K\]
getting
\begin{multline*}0= \int_{\Omega_T} (u_m)_t\phi_K  \left[T_k(u_m)-T_k(u)\right]dxdt\\+\int_{\Omega_T} \phi_K \omega_2 |\nabla u_m|^{p-2}\nabla u_m [\nabla T_k(u_m)-\nabla T_k(u)]dxdt\\
+\int_{\Omega_T}\omega_2 |\nabla u_m|^{p-2}\nabla u_m[T_k(u_m)-T_k(u)]\nabla\phi_K dxdt\\-\int_{\Omega_T} g_m[T_k(u_m)-T_k(u)]\phi_K dxdt\\
=J^m_1+J^m_2+J^m_3+J^m_4,
\end{multline*}  
where we show that $\lim_{m\to\infty}(
J^m_1+J^m_3+J^m_4)=0$. Then the convergence of~$J^m_2$ to zero follows and implies $E_3^m\to 0$.

We deal with $J^m_1$ and $J^m_4$ in the similar way. We note that either sequence $((u_m)_t)_m$ or $(g_m)_m$ are bounded sequences in $ L^{p'}(0,T; W^{-1,p'}_{(\omega_1',\omega_2')}(\Omega))$. Therefore, Proposition~\ref{prop:boc-mur-appl} implies that up to a subsequence \[T_k(u_m)-T_k(u)\xrightarrow[m\to\infty]{} 0\quad\text{strongly in}\quad L^{p}(0,T; L^{p}_{\omega_1,loc}(\Omega)),\] as we have (W5)  and~\eqref{emb:sp}. Then $J^m_1,J^m_4\to 0$ as $m\to \infty$. As for $J^m_3$, we apply the H\"older inequality, to get
\[\begin{split} J^m_3&= \int_{\Omega_T}\omega_2 |\nabla u_m|^{p-2}\nabla u_m[T_k(u_m)-T_k(u)]\nabla\phi_K dxdt \\
&= \int_{U_T}\omega_2 |\nabla u_m|^{p-2}\nabla u_m[T_k(u_m)-T_k(u)]\nabla\phi_K dxdt \\
%&\leq \int_0^T\left( \int_{U}[T_k(u_m)-T_k(u)]^q dx \right)^\frac{1}{q}\cdot\\&\qquad\qquad\qquad\cdot\left( \int_{U }\omega_2 |\nabla u_m|^{p} dx \right)^\frac{p-1}{p}\left( \int_{U }\left[\omega_2^{\frac{1}{p}}|\nabla\phi_K| \right]^\frac{qp}{q-p} dx \right)^\frac{q-p}{qp}\,dt \\ 
&\leq const\left[ \int_0^T\left(\int_{U}[T_k(u_m)-T_k(u)]^q dx \right)^\frac{p}{q}dt\right]^\frac{1}{p} \cdot\\&\qquad\qquad\qquad\cdot\left[ \int_0^T  \int_{U}\omega_2 |\nabla u_m|^{p} dx\,dt\right]^\frac{p-1}{p}\left( \int_{U}\omega_2^{\frac{q}{q-p}}dx\right)^\frac{q-p}{qp},
\end{split}\]  
where $c_H>0$, $U\subset\subset\Omega$ such that $\mathrm{supp}\phi_K\subset(0,T)\times U$, and $q$ comes from (W6).
Then \begin{equation}
\label{Jm3}
J^m_3\xrightarrow[m\to\infty]{}0.
\end{equation}

 Indeed,  by Proposition~\ref{prop:boc-mur-appl} we obtain $T_k(u_m)-T_k(u)\to 0$ strongly in~$L^p(0,T;L^s_{\omega_1}(U))$. Notice that~(W6) ensures that there exists $q$ such that  
\[L^p(0,T;L^s_{\omega_1}(U))\subset L^p(0,T;L^q(U)).\]
Moreover, the weak convergence of $(u_m)$ in $L^{p}(0,T;W^{1,p}_{(\omega_1,\omega_2),0}(\Omega))$ implies its uniform boundedness in this space (up to a subsequence), thus $\int_{U_T}\omega_2 |\nabla u_m|^{p} dxdt<C$, with a constant $C$ independent of $m$.
Finally,  $\int_{U }\omega_2^\frac{q}{q-p}dx<\infty$ due to~(W6). Therefore, we have~\eqref{Jm3}.

As $J^m_1+J^m_2+J^m_3+J^m_4=0$ and $\lim_{m\to\infty}(
J^m_1+J^m_3+J^m_4)=0$, then also  
 \begin{equation}
\label{eq:dlaEm}
 J^m_2=\int_{\Omega_T} \phi_K \omega_2 |\nabla u_m|^{p-2}\nabla u_m [\nabla T_k(u_m)-\nabla T_k(u)]dxdt\xrightarrow[m\to\infty]{}0.
 \end{equation}
Therefore, in~\eqref{Emsplit} we have $E_3^m\to 0$ and, hence,~\eqref{Em} and~\eqref{end}, which ends the proof.
\end{proof}

Existence of the solution to the~truncated problem is a consequence of the following result.

\begin{theo}[{\cite[Theorem~3.1]{isazg1}}]
 \label{thm:extrun} Let  $2<p<\infty$, $\Omega\subseteq\rn$ be an open subset, $f\in L^2(\Omega)$ and $\omega_1,\omega_2$ satisfy~(W1)-(W5).

 There exists  $\lambda_0=\lambda_0(p,N,\omega_1,\omega_2)>0$, such that for all $\lambda\in(0,\lambda_0)$ the parabolic problem
\begin{equation*}%\label{eq:maintrun}
\left\{\begin{array}{ll}
 u_t-\dL u=  \lambda W(x)|u|^{p-2}u & x\in\Omega,\\
 u(x,0)=f(x)& x\in\Omega,\\
 u(x,t)=0& x\in\partial\Omega,\ t>0,\\
\end{array}\right.
\end{equation*} 
where $W:\Omega\to\R_{+}$ is
such that
\[W(x)\leq \min\{m,\omega_1(x)\}\]
 with a certain $m\in\R_+$, has a global weak solution $u \in L^p(0,T; W_{(\omega_1,\omega_2),0}^{1,p}(\Omega)),$ such that $
u_t  \in L^{p'}(0,T; {W^{-1,p'}_{(\omega_1',\omega_2')}(\Omega)}),$ i.e.
\begin{equation*}
%\label{eq:mainweaktrun}
\int_{\Omega_T}\left(  u_t\xi+\omega_2|\nabla u|^{p-2} \nabla u \nabla \xi +\lambda W(x) |u|^{p-2}u \xi\right)dx\,dt=0,
\end{equation*}
holds for each $\xi\in L^p(0,T; W_{(\omega_1,\omega_2),0}^{1,p}(\Omega))$. Moreover, $u\in   L^\infty(0,T; L^2(\Omega_T))$.
\end{theo}

\begin{rem}\rm In our previous paper~\cite{isazg1} another embedding result was used, namely~\cite[Proposition~2.1]{anh}. It can be easily checked that the proof therein holds true as well, when we assume~(W5) instead of that one.\end{rem}

\section{Proof of the main result}\label{sec:main}

The main idea of the proof is to consider a truncated problem
\[(u_m)_t-\dL u_m =  \lambda T_m(\omega_1) |u_m|^{p-2}u_m ,\]
where $T_m$ is the truncation defined in~\eqref{Tk}, and then pass to the limit with $m\to\infty$ using the auxiliary compactness results of the previous section.
 
\begin{proof}[Proof of Theorem~\ref{theo:main}]  We consider $u_m$ --- the solution to the truncated problem
\begin{equation}\label{eq:mtrunc}\left\{\begin{array}{ll}
 w_t-\dL w=  \lambda T_m(\omega_1) |w|^{p-2}w & x\in\Omega\\
w(x,0)=f(x)& x\in\Omega\\
w(x,t)=0& x\in\partial\Omega,\ t>0,\\
\end{array}\right. 
\end{equation}
where  $T_m$ is given by~\eqref{Tk}.  Due to Theorem~\ref{thm:extrun} there exists a~solution $u_m$ to the problem~\eqref{eq:mtrunc} such that \[u_m \in L^p(0,T; W^{1,p}_{(\omega_1,\omega_2),0}(\O))\cap  L^\infty(0,T;  L^2(\Omega )),\quad (u_m)_t \in L^{p'}( 0,T; W^{-1,p'}_{(\omega_1,\omega_2)}(\Omega )).\] 

\textbf{A priori estimate.} To pass to the limit with $m\to \infty,$ we need to obtain a priori estimate. In order to get it, we test the~problem~\eqref{eq:mtrunc} by $u_m$ getting
\begin{multline*}
\frac{1}{2}\frac{d}{dt}\|u_m\|^2_{L^2(\Omega)}+\int_{\Omega } \omega_2|\nabla u_m|^pdx=\lambda\int_{\Omega } T_m(\omega_1)|u_m|^p\,dx\leq\\\leq \lambda\int_{\Omega } \omega_1 |u_m|^p\,dx
\leq \frac{\lambda}{K}\int_{\Omega } \omega_2|\nabla u_m|^pdx,
\end{multline*}
where the last inequality is allowed due to the Hardy inequality~\eqref{eq:hardyintro}. Note that the density of Lipschitz and compactly supported functions in $W^{1,p}_{(\omega_1,\omega_2),0}(\O))$ is given by Fact~\ref{factemb}. Therefore, \[
\frac{1}{2}\frac{d}{dt}\|u_m\|^2_{L^2(\Omega)}+\left(1-\frac{\lambda}{K}\right)\int_{\Omega } \omega_2|\nabla u_m|^pdx\leq 0.
\]

Note that \[
\int_0^T\frac{d}{dt}\|u_m\|^2_{L^2(\Omega)}dt =\|u_m  (\cdot,T)\|_{L^2(\Omega)}^2- \|f\|_{L^2(\Omega)}^2.\]

Summing up, we obtain
\begin{multline}\label{apriori}
\frac{1}{2}\| u_m (\cdot,T)\|_{L^2(\Omega)}^2+ \left(1-\frac{\lambda}{K}\right)\int_0^T\|\nabla u_m (\cdot,t)\|^p_{L_{\omega_2}^p(\Omega )}dt  \leq  %\frac{1}{2}\|u_m  (\cdot,0)\|_{L^2(\Omega)}^2=
\frac{1}{2}\|f \|_{L^2(\Omega)}^2. \end{multline}

\textbf{Convergence.} In particular, the above a priori estimate implies  
\begin{itemize}
\item $(u_m)_{m\in\n}$ is bounded in $ L^\infty(0,T;  L^2(\Omega ))$;
\item $(u_m)_{m\in\n}$ is bounded in $ L^p( 0,T; W_{(\omega_1,\omega_2),0}^{1,p}(\Omega)).$
\end{itemize}
Thus, there exists a function $u \in L^p(0,T; W_{(\omega_1,\omega_2),0}^{1,p}(\Omega ))\cap  L^\infty(0,T;  L^2(\Omega ))$ with  $u_t \in L^{p'}( 0,T; W_{(\omega_1,\omega_2)}^{-1,p'}(\Omega))$, such that and up to a subsequence, we have
\begin{eqnarray}
u_m\xrightharpoonup[m\to\infty]{\ \ *\ \ } u& \mathrm{in}&   L^\infty(0,T;  L^2(\Omega )),\label{eq:ulimit}\\
u_m\xrightharpoonup[m\to\infty]{\ \ \ \ \ } u& \mathrm{in}&   L^p( 0,T; W_{(\omega_1,\omega_2),0}^{1,p}(\Omega )).\nonumber
\end{eqnarray} 

We know that for each $\xi\in L^p(0,T; W_{(\omega_1,\omega_2),0}^{1,p}(\Omega))$ the following equality holds
\begin{equation} \label{eq:weakum}
\int_{\Omega_T}\left(  (u_m)_t\xi+\omega_2|\nabla u_m|^{p-2} \nabla u_m \nabla \xi +\lambda T_m(\omega_1) |u_m|^{p-2}u_m \xi\right)dx\,dt=0.
\end{equation}

\textbf{Identification of the limit function $u$.} We have to show that the limit  function $u$ from~\eqref{eq:ulimit} is the weak solution to~\eqref{eq:main}, i.e.
\begin{equation} \label{u:weak}
\int_{\Omega_T}\left(  u_t\xi+\omega_2|\nabla u|^{p-2} \nabla u \nabla \xi +\lambda \omega_1 |u|^{p-2}u \xi\right)dx\,dt=0 
\end{equation}
holds for each $\xi\in L^p(0,T; W_{(\omega_1,\omega_2),0}^{1,p}(\Omega))$. 

Let us note that the integral above is well--defined within this class, in particular $ L^p(0,T; W_{(\omega_1,\omega_2),0}^{1,p}(\Omega))\subset L^2(\Omega_T)$. The weak convergence of gradients is not enough to pass to the limit with  $\int_{Q_T}\omega_2|\nabla u_m|^{p-2} \nabla u_m \nabla \xi$. Thus, the first step is to get strong convergence of~gradients. We follow the spirit of~Boccardo and Murat to obtain a strong convergence  of the gradients of trucations and apply it in~\eqref{eq:weakum} splitted into
\begin{equation}
\label{sumAm}\begin{split}
0=&\int_{\Omega_T}   (u_m)_t\xi dx\,dt+\int_{\Omega_T\cap \{|u_m|\leq k\}} \omega_2|\nabla u_m|^{p-2} \nabla u_m \nabla \xi dx\,dt+\\
&+\int_{\Omega_T\cap \{|u_m|> k\}} \omega_2|\nabla u_m|^{p-2} \nabla u_m \nabla \xi dx\,dt+\int_{\Omega_T} \lambda T_m(\omega_1) |u_m|^{p-2}u_m \xi dx\,dt\\&=A_1^m+A_2^{k,m}+A_3^{k,m}+A_4^m,
\end{split}
\end{equation}
where we prove that $A_1^m,A_2^{k,m},A_4^m$ converges to the desired quantities to retrieve~\eqref{u:weak} in the limit, whereas $A_3^{k,m}\to 0$.

The convergence of 
\begin{equation}
\label{A1}A_1^m\xrightarrow[m\to\infty]{}\int_{\Omega_T} u_t\xi\,dxdt\end{equation} can be obtained by integrating by~parts and by the Lebesgue's Monotone Convergence Theorem since $(u_m)_m$ is a~nondecreasing sequence.

Our aim now is to show that
\begin{equation}
\label{A2}\begin{split}
\limsup_{k\to\infty}\lim_{m\to\infty}A_2^{k,m}%&= \limsup_{k\to\infty}\lim_{m\to\infty}\int_{\Omega_T}  \omega_2|\nabla T_k(u_m)|^{p-2} \nabla T_k(u_m) \nabla \xi dx\,dt\\
&=\int_{\Omega_T}  \omega_2|\nabla  u|^{p-2} \nabla  u  \nabla \xi dx\,dt.\end{split}
\end{equation} For this we use Proposition~\ref{prop:boc-mur} implying
\[ \nabla T_k(u_m)\xrightarrow[m\to\infty]{} \nabla T_k(u)\quad \mathrm{in}\quad L^{p}\left(0,T; (L^{p}_{\omega_2}(U))^N\right).\] Its assumptions are satisfied, because besides the weak convergence of functions, we have
\begin{equation}
\label{gmconv}g_m=\lambda\omega_1 |u_m|^{p-2}u_m\xrightarrow[m\to\infty]{} \lambda\omega_1 |u|^{p-2}u=g
\end{equation}
 in $L^{p'}(0,T; W^{-1,p'}_{(\omega_1',\omega_2')}(\Omega))$. To justify this we apply the Aubin--Lions Lemma (Theorem~\ref{AubinLionsLemmarefl}). Since we assume (W5) and we know~\eqref{emb:chain2}, we have \[W^{1,p}_{(\omega_1,\omega_2),0}(U)\subset\subset L^{p}_{ \omega_1}(U)\subset W^{-1,p'}_{(\omega_1',\omega_2')}(\Omega).\] 
 Then we infer that $u_m\to u$ strongly in $L^p(0,T;L^p_{\omega_1}(U))$. Strongly convergent sequence has a subsequence convergent almost everywhere. If it is necessary, we pass to such subsequence, but we do not change the notation. Note that 
\begin{multline*}\|g_m\|^{p'}_{ L^{p'}( 0,T; L_{ \omega_1' }^{ p'}(\Omega ))}=\lambda\int_{\Omega_T}\omega_1'\left|\omega_1|u_m|^{p-1}\right|^\frac{p}{p-1}dxdt=\\
=\lambda\int_{\Omega_T}\omega_1^{-\frac{1}{p-1}} \omega_1^{ \frac{p}{p-1}}|u_m|^{p } dxdt=\lambda\int_{\Omega_T}\omega_1 |u_m|^{p } dxdt<\infty \end{multline*} 
and thus\[g_m\in L^{p'}( 0,T; L_{ \omega_1' }^{ p'}(U))\subset L^{p'}( 0,T; W_{(\omega_1',\omega_2')}^{-1,p'}(\Omega )).\]
According to the Brezis--Lieb Lemma (Corollary~\ref{coro:BrezisLiebLemma}) the strong convergence of $u_m\to u$  in $L^p(0,T;L^p_{\omega_1}(U))$ implies the strong convergence 
\[\omega_1 |u_m|^{p-2}u_m\to \omega_1|u|^{p-2}u\quad\mathrm{ in}\ L^{p'}(0,T;L^{p'}_{\omega_1'}(U)),\]
which entails strong convergence $g_m \to g$ in $L^{p'}(0,T;L^{p'}(0,T;L^{p'}_{\omega_1'}(U))$
and in~turn also~\eqref{gmconv}. This finishes the case of $\lim_{m\to\infty}A_2^m$. Limit when $k\to\infty$ results from the Lebesgue Monotone Convergence  Theorem and a priori estimate~\eqref{apriori}.
 
To pass to the limit 
\begin{equation}
\label{A3}\begin{split}
\limsup_{k\to\infty}\limsup_{m\to\infty}A_3^{k,m} 
&=0.\end{split}
\end{equation} we notice first that the H\"older  inequality implies that $A_3^m\leq s(k),$ with a~certain constant $s$ depending on $k$. Indeed,
 \begin{multline*}
 A_3^m=\int_{\Omega_T\cap \{|u_m|> k\}} \omega_2|\nabla u_m|^{p-2} \nabla u_m \nabla \xi dx\,dt\leq \\\leq\left(\int_{\Omega_T\cap \{|u_m|> k\}} \omega_2|\nabla u_m|^{p}dx\,dt\right)^\frac{p-1}{p}\left(\int_{\Omega_T\cap \{|u_m|> k\}} \omega_2| \nabla \xi|^p dx\,dt\right)^\frac{1}{p}\\
 \leq const \left(\int_{\Omega_T\cap \{|u|> k\}} \omega_2| \nabla \xi|^p dx\,dt\right)^\frac{1}{p}=s(k).\end{multline*}
Note that the integral on the right--hand side above is finite even for $k=0$ and that the sequence $(u_m)$ is nondecreasing (and thus $\{|u_m|> k\}\subset\{|u |> k\}$). Moreover, $s(k)\to 0$ when $k\to\infty$. Thus, we have~\eqref{A3}.

To complete the analysis of~\eqref{sumAm} we need to show the limit of $A_4^m$, namely 
\begin{equation}
\label{A4}\begin{split}
\limsup_{k\to\infty}\limsup_{m\to\infty}A_4^{m} 
= \int_{\Omega_T} \lambda  \omega_1  |u |^{p-2}u  \xi dx\,dt.\end{split}
\end{equation} We have
\[\begin{split} A_4^m-\int_{\Omega_T} & \omega_1 |u |^{p-2}u  \xi dx\,dt=\\
 =&\int_{\Omega_T} T_m(\omega_1) |u_m|^{p-2}u_m \xi dx\,dt-\int_{\Omega_T}  \omega_1 |u |^{p-2}u  \xi dx\,dt=\\
 =&\int_{\Omega_T} (  T_m(\omega_1) -\omega_1)|u_m|^{p-2}u_m \xi dx\,dt\\  
&+ \int_{\Omega_T} (|u_m|^{p-2}u_m-|u|^{p-2}u)\omega_1   \xi dx\,dt=\\
=&B_1^m+B_2^m,\end{split}
\]
where we show that both $B_1^m$ and $B_2^m$ tend to zero with $m\to\infty$. 

To deal with $B_1^m$ we recall that  $(|  u_m|^{p-2}  u_m)_m$ is bounded in~$ L^{p'}( 0,T; W_{(\omega_1',\omega_2')}^{-1,p'}(\Omega ))$ (cf. the case of~$A_2^m$), while $T_m(\omega_1)\nearrow\omega_1$, so the Lebesgue Monotone Convergence Theorem implies $B_1^m\to 0$ as $m\to\infty$. 

Let us concentrate on $B_2^m$. We have
\begin{multline*}|B_2^m|\leq  \left(\int_{\Omega_T} \left||u_m|^{p-2}u_m-|u|^{p-2}u\right|^\frac{p}{p-1}\omega_1 dxdt\right)^\frac{p-1}{p}   \left(\int_{\Omega_T}  \omega_1   |\xi|^p dx\,dt\right)^\frac{1}{p}.
\end{multline*}
We employ  the Brezis--Lieb Lemma (Corollary~\ref{coro:BrezisLiebLemma}) to get \[\omega_1 |u_m|^{p-2}u_m\to \omega_1|u|^{p-2}u_m\quad\mathrm{ in}\ L^{p'}(0,T;L^{p'} (U)),\]
leading to
\[  |u_m|^{p-2}u_m\to  |u|^{p-2}u_m\quad\mathrm{ in}\ L^{p'}(0,T;L^{p'}_{\omega_1}(U)).\]
which implies that $B_2^m\to 0$ as $m\to\infty$.

When we pass to the limit in~\eqref{sumAm} according to~\eqref{A1}, \eqref{A2}, \eqref{A3}, and \eqref{A4}, we get~\eqref{u:weak} and, thus, we conclude that $u$ is the desired weak solution.
\end{proof}

\section*{Appendix}\label{sec:auxtools}

For the sake of completeness we recall the general analytic tools necessary in~our approach. 
\begin{theo}[The Vitali Convergence Theorem]\label{theo:VitConv} Let $(X,\mu)$ be a positive measure space. If $\mu(X)<\infty $, $\{f_{n}\}$ is uniformly integrable,   $f_{n}(x)\to f(x)$ a.e.  and $|f(x)|<\infty $  a.e. in $X$, then  $f\in  {L}^1_\mu(X)$
and  $f_{n}(x)\to f(x)$ in  ${L}^1_\mu(X)$.
\end{theo}

For the Aubin--Lions Lemmas we refer e.g. to~\cite{simon}.  
\begin{theo}[The Aubin Lions Lemma~1] 
 \label{AubinLionsLemmarefl} Suppose $1<p<\infty$, $X,B,Y$ are the Banach spaces, $X\subset\subset B\subset Y$, $F$ is bounded in $L^p(0,T;X)$ and~relatively compact in $L^p(0,T;Y)$ then $F$ is relatively compact in $L^p(0,T;B)$.
\end{theo}
 \begin{theo}[The Aubin Lions Lemma~2]\label{AubinLionsLemma} Suppose $1\leq p<\infty$, $X,B,Y$ are the Banach spaces, $X\subset\subset B\subset Y$.   If $F$ is bounded in $L^p(0,T;X)$ and~$\frac{dF}{dt}$ is bounded in~$L^r(0,T;Y)$, where $r>1$,  then $F$ is relatively compact in~$L^p(0,T;B)$.
\end{theo}

For the Brezis Lieb Lemma we refer to~\cite{brezlieb}.
\begin{theo}[The Brezis Lieb Lemma]
 \label{BrezisLiebLemma} Suppose $\Omega\subset\rn$, $1\leq p<\infty$, and $\mu\geq 0$ is a Radon measure. If $f_n\to f$ a.e. in $\Omega$ and $(f_n)_n$ is bounded in ${L^p_\mu(\Omega)}$, then the following limit exists \[\lim_{n\to\infty}\left(\|f_n\|_{L^p_\mu(\Omega)}^p-\|f-f_n\|_{L^p_\mu(\Omega)}^p\right)= \|f \|_{L^p_\mu(\Omega)}^p\]
and the equality holds.
\end{theo}

We have the following corollary of the above theorem.
\begin{coro}\label{coro:BrezisLiebLemma} Suppose $\Omega\subset\rn$, $1\leq p<\infty$, and $\omega_1:\Omega\to\r\cup\{0\}$ is measurable. If $u_m\to u$ strongly in $L^p(0,T; L^p_{\omega_1}(\Omega))$, then \[\omega_1 |u_m|^{p-2}u_m\to\omega_1 |u|^{p-2}u\quad \mathrm{strongly\ in\ }L^{p'}(0,T; L^{p'}_{\omega_1'}(\Omega)).\]
\end{coro}\begin{proof}%[Proof of Corollary~\ref{coro:BrezisLiebLemma}]
If $u_m\to u$ strongly in $L^p(0,T; L^p_{\omega_1}(\Omega))$ and a.e. in~$\Omega$, then  Theorem~\ref{BrezisLiebLemma}  yields that
\[\int_{\Omega_T}\omega_1 |u_m|^p\,dx\,dt\to\int_{\Omega_T}\omega_1 |u |^p\,dx\,dt.\]
Equivalently,
\[\int_{\Omega_T}\omega_1 \left||u_m|^{p-2}u_m\right|^\frac{p}{p-1}\,dx\,dt\to\int_{\Omega_T}\omega_1 \left||u |^{p-2}u \right|^\frac{p}{p-1}\,dx\,dt,\]
which, once again by  Theorem~\ref{BrezisLiebLemma}, implies
\[ |u_m|^{p-2}u_m\to  |u|^{p-2}u\quad \mathrm{strongly\ in\ }L^{p'}(0,T; L^{p'}_{\omega_1 }(\Omega)).\]
When we observe that
\[\begin{split}&\int_{\Omega_T}\omega_1 \left(|u_m|^{p-1}-|u|^{p-1} \right)^\frac{p}{p-1}\,dx\,dt\\&=\int_{\Omega_T}\omega_1' \left(\omega_1 |u_m |^{p-1}-\omega_1 |u |^{p-1} \right)^\frac{p}{p-1}\,dx\,dt,\end{split}\]
we conclude that
\[\omega_1 |u_m|^{p-2}u_m\to\omega_1 |u|^{p-2}u\quad \mathrm{strongly\ in\ }L^{p'}(0,T; L^{p'}_{\omega_1'}(\Omega)).\]
\end{proof}

\bibliographystyle{plain}
\bibliography{bib}

\def\ocirc#1{\ifmmode\setbox0=\hbox{$#1$}\dimen0=\ht0 \advance\dimen0
  by1pt\rlap{\hbox to\wd0{\hss\raise\dimen0
  \hbox{\hskip.2em$\scriptscriptstyle\circ$}\hss}}#1\else {\accent"17 #1}\fi}
  \def\cprime{$'$} \def\ocirc#1{\ifmmode\setbox0=\hbox{$#1$}\dimen0=\ht0
  \advance\dimen0 by1pt\rlap{\hbox to\wd0{\hss\raise\dimen0
  \hbox{\hskip.2em$\scriptscriptstyle\circ$}\hss}}#1\else {\accent"17 #1}\fi}
  \def\ocirc#1{\ifmmode\setbox0=\hbox{$#1$}\dimen0=\ht0 \advance\dimen0
  by1pt\rlap{\hbox to\wd0{\hss\raise\dimen0
  \hbox{\hskip.2em$\scriptscriptstyle\circ$}\hss}}#1\else {\accent"17 #1}\fi}
  \def\cprime{$'$}
\begin{thebibliography}{10}

\bibitem{AdChaRa}
Adimurthi, N.~Chaudhuri, and M.~Ramaswamy.
\newblock An improved {H}ardy-{S}obolev inequality and its application.
\newblock {\em Proc. Amer. Math. Soc.}, 130(2):489--505 (electronic), 2002.

\bibitem{anh}
C.~T. Anh and T.~D. Ke.
\newblock On quasilinear parabolic equations involving weighted
  {$p$}-{L}aplacian operators.
\newblock {\em NoDEA Nonlinear Differential Equations Appl.}, 17(2):195--212,
  2010.

\bibitem{bargold}
P.~Baras and J.~Goldstein.
\newblock The heat equation with a singular potential.
\newblock {\em Trans. Amer. Math. Soc.}, 284(1):121--139, 1984.

\bibitem{bbdg}
A.~Blanchet, M.~Bonforte, J.~Dolbeault, G.~Grillo, and J.-L. V{\'a}zquez.
\newblock Hardy-{P}oincar\'e inequalities and applications to nonlinear
  diffusions.
\newblock {\em C. R. Math. Acad. Sci. Paris}, 344(7):431--436, 2007.

\bibitem{blanchet_09}
A.~Blanchet, M.~Bonforte, J.~Dolbeault, G.~Grillo, and J.-L. V{\'a}zquez.
\newblock Asymptotics of the fast diffusion equation via entropy estimates.
\newblock {\em Arch. Ration. Mech. Anal.}, 191(2):347--385, 2009.

\bibitem{boc-ors1}
L.~Boccardo, A.~Dall'Aglio, T.~Gallou{\"e}t, and L.~Orsina.
\newblock Nonlinear parabolic equations with measure data.
\newblock {\em J. Funct. Anal.}, 147(1):237--258, 1997.

\bibitem{boc-ors2}
L.~Boccardo, T.~Gallou{\"e}t, and L.~Orsina.
\newblock Existence and nonexistence of solutions for some nonlinear elliptic
  equations.
\newblock {\em J. Anal. Math.}, 73:203--223, 1997.

\bibitem{boccardo}
L.~Boccardo and F.~Murat.
\newblock Almost everywhere convergence of the gradients of solutions to
  elliptic and parabolic equations.
\newblock {\em Nonlinear Anal.}, 19(6):581--597, 1992.

\bibitem{bocmurpuel}
L.~Boccardo, F.~Murat, and J.-P. Puel.
\newblock Existence of bounded solutions for nonlinear elliptic unilateral
  problems.
\newblock {\em Ann. Mat. Pura Appl. (4)}, 152:183--196, 1988.

\bibitem{sharp}
M.~Bonforte, J.~Dolbeault, G.~Grillo, and J.~L. V{\'a}zquez.
\newblock Sharp rates of decay of solutions to the nonlinear fast diffusion
  equation via functional inequalities.
\newblock {\em Proc. Natl. Acad. Sci. USA}, 107(38):16459--16464, 2010.

\bibitem{weightmatteo1}
M.~Bonforte, J.~Dolbeault, M.~Muratori, and B.~Nazaret.
\newblock Weighted fast diffusion equations (part i): Sharp asymptotic rates
  without symmetry and symmetry breaking in caffarelli-kohn-nirenberg
  inequalities.
\newblock {\em To Appear in Kinet. Rel. Mod.}, 2016.

\bibitem{weightmatteo2}
M.~Bonforte, J.~Dolbeault, M.~Muratori, and B.~Nazaret.
\newblock Weighted fast diffusion equations (part ii): Sharp asymptotic rates
  of convergence in relative error by entropy methods.
\newblock {\em To Appear in Kinet. Rel. Mod.}, 2016.

\bibitem{brezlieb}
H.~Br{\'e}zis and E.~Lieb.
\newblock A relation between pointwise convergence of functions and convergence
  of functionals.
\newblock {\em Proc. Amer. Math. Soc.}, 88(3):486--490, 1983.

\bibitem{BV}
H.~Brezis and J.-L. V{\'a}zquez.
\newblock Blow-up solutions of some nonlinear elliptic problems.
\newblock {\em Rev. Mat. Univ. Complut. Madrid}, 10(2):443--469, 1997.

\bibitem{bh}
R.~C. Brown and D.~B. Hinton.
\newblock Weighted interpolation and {H}ardy inequalities with some
  spectral-theoretic applications.
\newblock In {\em Proceedings of the {F}ourth {I}nternational {C}olloquium on
  {D}ifferential {E}quations ({P}lovdiv, 1993)}, pages 55--70. VSP, Utrecht,
  1994.

\bibitem{dol-tos}
J.~Dolbeault and G.~Toscani.
\newblock Fast diffusion equations: matching large time asymptotics by relative
  entropy methods.
\newblock {\em Kinet. Relat. Models}, 4(3):701--716, 2011.

\bibitem{fra-ser-ser}
B.~Franchi, R.~Serapioni, and F.~Serra~Cassano.
\newblock Approximation and imbedding theorems for weighted {S}obolev spaces
  associated with {L}ipschitz continuous vector fields.
\newblock {\em Boll. Un. Mat. Ital. B (7)}, 11(1):83--117, 1997.

\bibitem{gaap}
J.~P. Garc{\'{\i}}a~Azorero and I.~Peral~Alonso.
\newblock Hardy inequalities and some critical elliptic and parabolic problems.
\newblock {\em J. Differential Equations}, 144(2):441--476, 1998.

\bibitem{32}
J.~Goldstein and Q.~Zhang.
\newblock On a degenerate heat equation with a singular potential.
\newblock {\em J. Funct. Anal.}, 186(2):342--359, 2001.

\bibitem{kuf-opic}
A.~Kufner and B.~Opic.
\newblock How to define reasonably weighted {S}obolev spaces.
\newblock {\em Comment. Math. Univ. Carolin.}, 25(3):537--554, 1984.

\bibitem{OpKuf}
B.~Opic and A.~Kufner.
\newblock {\em Hardy-type inequalities}, volume 219 of {\em Pitman Research
  Notes in Mathematics Series}.
\newblock Longman Scientific \& Technical, Harlow, 1990.

\bibitem{simon}
J.~Simon.
\newblock Compact sets in the space {$L^p(0,T;B)$}.
\newblock {\em Ann. Mat. Pura Appl. (4)}, 146:65--96, 1987.

\bibitem{plap}
I.~Skrzypczak.
\newblock Hardy-type inequalities derived from {$p$}-harmonic problems.
\newblock {\em Nonlinear Anal.}, 93:30--50, 2013.

\bibitem{isazg1}
I.~Skrzypczak and A.~Zatorska-Goldstein.
\newblock Existence to degenerated parabolic problems with two weights via
  hardy inequality.
\newblock {\em submitted}, 2016.

\bibitem{vz}
J.-L. Vazquez and E.~Zuazua.
\newblock The {H}ardy inequality and the asymptotic behaviour of the heat
  equation with an inverse-square potential.
\newblock {\em J. Funct. Anal.}, 173(1):103--153, 2000.

\bibitem{xiang}
C.-L. Xiang.
\newblock Asymptotic behaviors of solutions to quasilinear elliptic equations
  with critical {S}obolev growth and {H}ardy potential.
\newblock {\em J. Differential Equations}, 259(8):3929--3954, 2015.

\end{thebibliography}

\end{document}